\documentclass[submission%
]{dmtcs}

\usepackage{etex}

\usepackage[latin1]{inputenc}
\usepackage{subfigure}

\usepackage{subfigure}
\usepackage{xcolor}           
\usepackage{enumitem}         
\usepackage{tikz}
\usetikzlibrary{fit,calc,positioning,decorations.pathreplacing,matrix, shapes}
\usetikzlibrary{trees}

\usepackage{todonotes}

\usepackage{xspace}

\usepackage{amsfonts, amssymb, amsmath}
\usepackage{bm}               
\usepackage{stmaryrd} 

\newtheorem{thm}{Theorem}[section]
 \newtheorem{defi}[thm]{Definition}
 \newtheorem{exple}[thm]{Example}
 
  \newtheorem{prop}[thm]{Proposition}
  \newtheorem{lem}[thm]{Lemma}

   \newtheorem{rque}[thm]{Remark}

\usepackage{xargs}

\newcommand{\SG}{\mathfrak{S}} 
\newcommand{\Vect}{\operatorname{Vect}}
\newcommand{\W}{\mathcal{W}}
\newcommandx{\WP}[1][1=p]{\{\W_1,\cdots,\W_{#1}\}} 
\renewcommand{\H}{\mathcal{H}}
\newcommandx{\HP}[1][1=p]{\{\H_1,\cdots,\H_{#1}\}} 
\newcommandx{\WiP}[2][1=i,2=p]{\{\W_{#1,1},\cdots,\W_{#1,#2}\}} 
\renewcommandx{\i}[2]{\llbracket #1,#2 \rrbracket}
\newcommand{\NAT}{\mathcal{N\!AT}} 
\newcommandx{\NATdk}[2][1=d,2=k]{\NAT_{#1,#2}} 
\newcommand{\n}{n} 
\renewcommand{\d}{d} 
\renewcommand{\k}{k} 
\newcommand{\B}{B} 
\newcommand{\M}{M} 
\renewcommand{\O}{O} 
\newcommand{\BT}{\mathcal{BT}} 
\newcommand{\N}{T} 
\newcommand{\NOT}{\mathcal{NOT}} 
\newcommand{\GFN}{\mathfrak{N}} 
\newcommand{\GFH}{\mathfrak{H}} 
\newcommand{\LGFN}{\mathfrak{G}} 
\newcommandx{\GFHdk}[2][1=d,2=k]{\mathfrak{H}_{#1,#2}} 
\newcommandx{\GFNdk}[2][1=d,2=k]{\mathfrak{N}_{#1,#2}} 
\newcommandx{\LGFNdk}[2][1=d,2=k]{\mathfrak{G}_{#1,#2}} 
\newcommand{\gs}{\wi_1 \times \dots \times \wi_d} 
\newcommand{\w}{w} 
\newcommand{\wi}{\w} 
\newcommand{\dir}{\pi} 
\newcommandx{\edir}[2][1=d,2=k]{\Pi_{#1,#2}} 
\newcommand{\V}{\mathcal{V}}
\newcommand{\LV}{\V_L} 
\newcommand{\RV}{\V_R} 
\newcommand{\inv}{\operatorname{Inv}} 
\newcommand{\imaj}{\operatorname{iMaj}} 
\newcommand{\std}{\operatorname{Std}} 
\newcommand{\DV}{\V_{\dir}} 
\newcommandx{\GFD}[1][1=d]{\mathfrak{D}_{#1}} 
\newcommandx{\DTp}[1][1=p]{\mathcal{DT}_{d,#1}} 
\newcommandx{\NATB}[1][1=\B]{\NAT\left(#1\right)} 
\newcommandx{\NATM}[1][1=\M]{\NAT\left(#1\right)} 
\newcommandx{\EL}{\mathcal{E}_L} 
\newcommandx{\ER}{\mathcal{E}_R} 
\newcommand{\E}{\mathcal{E}} 
\newcommand{\LO}{\mathcal{L}_0} 
\newcommand{\RO}{\mathcal{R}_0} 
\newcommandx{\nat}[2][1=d,2=k]{NAT$_{#1,#2}$\xspace}
\newcommand{\QQ}{\operatorname{\mathbb{Q}}}
\newcommand{\X}{X}

\newcommand{\ft}{\mathcal{T}_4}
\newcommand{\BNAT}{\operatorname{\mathbb{M}}} 
\newcommand{\BSG}{\operatorname{\mathbb{B}\SG}} 
\newcommand{\Bxy}{\operatorname{\mathbb{B}}} 

\newcommand{\NodeBT}[2]{
\begin{tikzpicture}[baseline=(current bounding box.base)]
  \filldraw[color=black] (7pt, 7pt) circle (2pt);
  \draw (7pt,7pt)--(0,0);
  \node at (-3pt,-3pt) {$#1$};
  \draw (7pt,7pt)--(14pt,0pt);
  \node at (17pt,-3pt) {$#2$};
\end{tikzpicture}}

\author[J.-C. Aval \and A. Boussicault \and B. Delcroix-Oger \and F. Hivert \and 	P. Laborde-Zubieta]{
		Jean-Christophe Aval
		\addressmark{1}
	\and 
		Adrien Boussicault
		\addressmark{1}
	\and 
		B\'{e}r\'{e}nice Delcroix-Oger
		\addressmark{2}
	\and 
		Florent Hivert
		\addressmark{3}
	\and
		Patxi Laborde-Zubieta
		\addressmark{1}
} 
  
\title{Non-ambiguous trees: new results and generalisation}
\address{
	\addressmark{1}
		LAboratoire Bordelais de Recherche en Informatique (UMR CNRS 5800),
		Universit\'e de Bordeaux, 33405 TALENCE \\
	\addressmark{2} 
		Institut de Math\'ematiques de Toulouse (UMR CNRS 5219),
		Universit\'e Paul Sabatier, 31062 TOULOUSE\\
	\addressmark{3}
		Laboratoire de Recherche en Informatique (UMR CNRS 8623)
                B\^atiment 650, Universit\'e Paris Sud 11, 91405 ORSAY CEDEX \\
}

\begin{document}
\maketitle
\begin{abstract}
\paragraph{Abstract.} 
We present a new definition of non-ambiguous trees (NATs) as labelled binary trees. We thus get a differential equation whose solution can be described combinatorially. This yield a new formula for the number of NATs. We also obtain $q$-versions of our formula. And we generalize NATs to higher dimension. \\

\paragraph{R\'esum\'e.} 
Nous introduisons une nouvelle d\'efinition des arbres non ambigus (NATs) en terme d'arbres binaires \'etiquet\'es. Nous en d\'eduisons une \'equation diff\'erentielle, dont les solutions peuvent \^etre d\'ecrites de mani\`ere combinatoire. Ceci conduit  \`a une nouvelle formule pour le nombre de NATs. Nous d\'emontrons aussi des $q$-versions des formules obtenues. Enfin, nous g\'en\'eralisons la notion de NAT
en dimension sup\'erieure.
\end{abstract}

\section*{Introduction}

Non-ambiguous trees (NATs for short) were introduced in a previous paper \cite{AvaBouBouSil14}.
We propose in the present article a sequel to this work.

Tree-like tableaux \cite{ABN} are certain fillings of Ferrers diagram, in simple bijection 
with permutations or alternative tableaux \cite{postnikov,viennot}. They are the subject of an intense
research activity in combinatorics, mainly because they appear as the key tools
in the combinatorial interpretation of the well-studied model of statistical mechanics
called PASEP: they naturally encode the states of the PASEP, together with
the transition probabilities through simple statistics \cite{CorteelWilliams}.

Among tree-like tableaux, NATs were defined as rectangular-shaped objects in \cite{AvaBouBouSil14}.
In this way, they are in bijection with permutation $\sigma=\sigma_1\,\sigma_2\,\dots\,\sigma_n$
such that the excedences ($\sigma_i>i$) are placed at the beginning of the word $\sigma$.
Such permutations were studied by Ehrenborg and Steingrimsson \cite{ES},
who obtained an explicit enumeration formula. Thanks to NATs, a bijective proof
of this formula was described in \cite{AvaBouBouSil14}. 

In the present work, we define NATs as labelled binary trees (see Definition \ref{def_nat}, which is equivalent
to the original definition). This new presentation allows us to obtain many new results about these objects.
The plan of the article is the following.\\
In Section \ref{nat}, we (re-)define NATs as binary trees whose right and left children are respectively labelled
with two sets of labels. We show how the generating series for these objects satisfies differential
equations (Prop. \ref{diff_eq_nat}), whose solution is quite simple and explicit (Prop. \ref{gen_ser_nat}). A combinatorial interpretation
of this expression involves the (new) notion of hooks in binary trees, linked to the notion of leaves in ordered trees. Moreover this expression yields a new 
formula for the number of NATs as a positive sum (see Theorem \ref{thm_p}), where Ehrenborg-Steingrimsson's formula is alternating.
To conclude with Section \ref{nat}, we obtain $q$-analogues of our formula, which are similar to those obtained
for binary trees in \cite{BjornerWachs,HivNovThib08} (see Theorem \ref{thm-q-hook}, the relevant statistics are either
the number of inversions or the inverse major index).\\
Section \ref{gnat} presents a generalisation of NATs in higher dimension. For any $k\le d$, we consider
NATs of dimension $(d,k)$, embedded  in $\mathbb{Z}^d$, and with edges of dimension $k$
\footnote{A definition in terms of labelled trees
is given in Subsection \ref{def_natdk}.}. 
The original case corresponds to dimension $(2,1)$. 
Our main result on this question is a differential equation satisfied by the generating series of these new objects.

This version of our work is an {\em extended abstract}; most proofs are only sketched or purely omitted.

\section{Non-ambiguous trees} \label{nat}

\subsection{Definitions} \label{definitions}

We recall that a \emph{binary tree} is a rooted tree whose vertices may have 
no child, or one left child, or one right child or both of them. 
The size of a binary tree is its number of vertices.
The empty binary tree, denoted by $\emptyset$, is the unique binary tree with 
no vertices. Having no child in one direction (left or right) is the same as
having an empty subtree in this direction.
We denote by $\BT$ the set of binary trees and by $\BT^*$ the set 
$\BT \setminus \{\emptyset\}$.
Given a binary tree $B$, we denote by $\LV(B)$ and $\RV(B)$ the set of left
children (also called left vertices) and the set of right children (also called
right vertices). We shall extend this notation to NATs.

\begin{defi} \label{def_nat} A \emph{non-ambiguous tree} (NAT) $T$ is a labelling of a
  binary tree $B$ such that :
\begin{itemize}
\item the left (resp. right) children are labelled from $1$ to $|\LV(B)|$
  (resp. $|\RV(B)|$), such that different left (resp. right) vertices have
  different labels. In other words, each left (right) label appears only once.
\item if $U$ and $V$ are two left (resp. right) children in the tree, 
such that $U$ is an ancestor of $V$, then the label of $U$ in $T$ is strictly greater 
than the label of $V$.
\end{itemize}
\end{defi}
The underlying tree of a non-ambiguous tree is called its \emph{shape}.  The
size $\n(T)$ of a NAT $T$ is its number of vertices. Clearly $\n(T) = 1
+ |\LV(T)| + |\RV(T)|$. It is sometimes useful to label the root as well. In
this case, it is considered as both a left and right child so that it carries
a pairs of labels, namely $(|\LV(T)|+1, |\RV(T)|+1)$. On pictures, to ease the
reading, we color the labels of left and right vertices in red and blue
respectively. 
Figure \ref{fig_exple_nat} shows an example of a NAT,
and illustrates the correspondence between the geometrical presentation
of \cite{AvaBouBouSil14} and Definition \ref{def_nat}.
The rectangle which contains the non-ambiguous
tree $T$ is of dimension $(\wi_L(T),\wi_R(T)) = (|\LV(T)|+1, |\RV(T)|+1)$.

\tikzset{every node/.style={inner sep=1pt,scale=0.8}}
\begin{figure} 
\begin{center}
\[T=\begin{tikzpicture}[baseline=(current bounding box.center),scale=0.35]
\node (r) at (0,8) {(\textcolor{red}{11},\textcolor{blue}{12})};
\node (11b) at (1,7) {\textcolor{blue}{11}};
\node (10b) at (-1,6) {\textcolor{blue}{10}};
\node (9b) at (0,5) {\textcolor{blue}{9}};
\node (8b) at (-1,4) {\textcolor{blue}{8}};
\node (7b) at (2,6) {\textcolor{blue}{7}};
\node (6b) at (0,3) {\textcolor{blue}{6}};
\node (5b) at (2,4){\textcolor{blue}{5}};
\node (4b) at (3,5) {\textcolor{blue}{4}};
\node (3b) at (4,4){\textcolor{blue}{3}};
\node (2b) at (3,1){\textcolor{blue}{2}};
\node (1b) at (5,3){\textcolor{blue}{1}};
\node (10r) at (-2,7){\textcolor{red}{10}};
\node (9r) at (1,5){\textcolor{red}{9}};
\node (8r) at (1,3){\textcolor{red}{8}};
\node (7r) at (3,3){\textcolor{red}{7}};
\node (6r) at (-2,5){\textcolor{red}{6}};
\node (5r) at (2,2){\textcolor{red}{5}};
\node (4r) at (-2,3){\textcolor{red}{4}};
\node (3r) at (-1,2){\textcolor{red}{3}};
\node (2r) at (-3,6){\textcolor{red}{2}};
\node (1r) at (-3,4){\textcolor{red}{1}};
\draw (2r)--(10r)--(r)--(11b)--(7b)--(4b)--(3b)--(1b);
\draw (7b)--(9r);
\draw (9r)--(5b);
\draw (5b)--(8r);
\draw (3b)--(7r);
\draw (7r)--(5r);
\draw (5r)--(2b);
\draw (10r)--(10b);
\draw (9b)--(10b)--(6r)--(1r);
\draw (6b)--(8b)--(4r);
\draw (3r)--(6b);
\draw (6r)--(8b);
\end{tikzpicture}
\hspace*{0.5cm}
\begin{tikzpicture}[baseline=(current bounding box.center),
                    scale=0.25,
                    every node/.style={inner sep=0pt,scale=0.7}]
\draw (0,0.75) node{\textcolor{blue}{12}};
\draw (1,0.75) node{\textcolor{blue}{11}};
\draw (2,0.75) node{\textcolor{blue}{10}};
\draw (3,0.75) node{\textcolor{blue}{9}};
\draw (4,0.75) node{\textcolor{blue}{8}};
\draw (5,0.75) node{\textcolor{blue}{7}};
\draw (6,0.75) node{\textcolor{blue}{6}};
\draw (7,0.75) node{\textcolor{blue}{5}};
\draw (8,0.75) node{\textcolor{blue}{4}};
\draw (9,0.75) node{\textcolor{blue}{3}};
\draw (10,0.75) node{\textcolor{blue}{2}};
\draw (11,0.75) node{\textcolor{blue}{1}};
\draw (-0.75,0) node{\textcolor{red}{11}};
\draw (-0.75,-1) node{\textcolor{red}{10}};
\draw (-0.75,-2) node{\textcolor{red}{9}};
\draw (-0.75,-3) node{\textcolor{red}{8}};
\draw (-0.75,-4) node{\textcolor{red}{7}};
\draw (-0.75,-5) node{\textcolor{red}{6}};
\draw (-0.75,-6) node{\textcolor{red}{5}};
\draw (-0.75,-7) node{\textcolor{red}{4}};
\draw (-0.75,-8) node{\textcolor{red}{3}};
\draw (-0.75,-9) node{\textcolor{red}{2}};
\draw (-0.75,-10) node{\textcolor{red}{1}};
\draw[gray!50] (0,0.5)--(0,-10.5);
\draw[gray!50] (1,0.5)--(1,-10.5);
\draw[gray!50] (2,0.5)--(2,-10.5);
\draw[gray!50] (3,0.5)--(3,-10.5);
\draw[gray!50] (4,0.5)--(4,-10.5);
\draw[gray!50] (5,0.5)--(5,-10.5);
\draw[gray!50] (6,0.5)--(6,-10.5);
\draw[gray!50] (7,0.5)--(7,-10.5);
\draw[gray!50] (8,0.5)--(8,-10.5);
\draw[gray!50] (9,0.5)--(9,-10.5);
\draw[gray!50] (10,0.5)--(10,-10.5);
\draw[gray!50] (11,0.5)--(11,-10.5);
\draw[gray!50] (-0.5,0)--(11.5,0);
\draw[gray!50] (-0.5,-1)--(11.5,-1);
\draw[gray!50] (-0.5,-2)--(11.5,-2);
\draw[gray!50] (-0.5,-3)--(11.5,-3);
\draw[gray!50] (-0.5,-4)--(11.5,-4);
\draw[gray!50] (-0.5,-5)--(11.5,-5);
\draw[gray!50] (-0.5,-6)--(11.5,-6);
\draw[gray!50] (-0.5,-7)--(11.5,-7);
\draw[gray!50] (-0.5,-8)--(11.5,-8);
\draw[gray!50] (-0.5,-9)--(11.5,-9);
\draw[gray!50] (-0.5,-10)--(11.5,-10);
\node (r) at (0,0) {$\bullet$};
\node (11b) at (12-11,0) {$\bullet$};
\node (10b) at (12-10,10-11) {$\bullet$};
\node (9b) at (12-9,10-11) {$\bullet$};
\node (8b) at (12-8,6-11) {$\bullet$};
\node (7b) at (12-7,0) {$\bullet$};
\node (6b) at (12-6,6-11) {$\bullet$};
\node (5b) at (12-5,9-11){$\bullet$};
\node (4b) at (12-4,0) {$\bullet$};
\node (3b) at (12-3,0){$\bullet$};
\node (2b) at (12-2,5-11){$\bullet$};
\node (1b) at (12-1,0){$\bullet$};
\node (10r) at (0,10-11){$\bullet$};
\node (9r) at (12-7,9-11){$\bullet$};
\node (8r) at (12-5,8-11){$\bullet$};
\node (7r) at (12-3,7-11){$\bullet$};
\node (6r) at (12-10,6-11){$\bullet$};
\node (5r) at (12-3,5-11){$\bullet$};
\node (4r) at (12-8,4-11){$\bullet$};
\node (3r) at (12-6,3-11){$\bullet$};
\node (2r) at (0,2-11){$\bullet$};
\node (1r) at (12-10,1-11){$\bullet$};
\draw (2r)--(10r)--(r)--(11b)--(7b)--(4b)--(3b)--(1b);
\draw (7b)--(9r);
\draw (9r)--(5b);
\draw (5b)--(8r);
\draw (3b)--(7r);
\draw (7r)--(5r);
\draw (5r)--(2b);
\draw (10r)--(10b);
\draw (9b)--(10b)--(6r)--(1r);
\draw (6b)--(8b)--(4r);
\draw (3r)--(6b);
\draw (6r)--(8b);
\end{tikzpicture}
\hspace*{0.5cm}
T_L=
\begin{tikzpicture}[baseline=(current bounding box.center),scale=0.35]
\node (10b) at (-1,6) {\textcolor{blue}{4}};
\node (9b) at (0,5) {\textcolor{blue}{3}};
\node (8b) at (-1,4) {\textcolor{blue}{2}};
\node (6b) at (0,3) {\textcolor{blue}{1}};
\node (10r) at (-2,7){(\textcolor{red}{6},\textcolor{blue}{5})};
\node (6r) at (-2,5){\textcolor{red}{5}};
\node (4r) at (-2,3){\textcolor{red}{4}};
\node (3r) at (-1,2){\textcolor{red}{3}};
\node (2r) at (-3,6){\textcolor{red}{2}};
\node (1r) at (-3,4){\textcolor{red}{1}};
\draw (2r)--(10r);
\draw (10r)--(10b);
\draw (9b)--(10b)--(6r)--(1r);
\draw (6b)--(8b)--(4r);
\draw (3r)--(6b);
\draw (6r)--(8b);
\end{tikzpicture}
\hspace{0.5cm}
T_R=
\begin{tikzpicture}[baseline=(current bounding box.center),scale=0.35]
\node (11b) at (1,7) {(\textcolor{red}{6},\textcolor{blue}{7})};
\node (7b) at (2,6) {\textcolor{blue}{6}};
\node (5b) at (2,4){\textcolor{blue}{5}};
\node (4b) at (3,5) {\textcolor{blue}{4}};
\node (3b) at (4,4){\textcolor{blue}{3}};
\node (2b) at (3,1){\textcolor{blue}{2}};
\node (1b) at (5,3){\textcolor{blue}{1}};
\node (9r) at (1,5){\textcolor{red}{4}};
\node (8r) at (1,3){\textcolor{red}{3}};
\node (7r) at (3,3){\textcolor{red}{2}};
\node (5r) at (2,2){\textcolor{red}{1}};
\draw (11b)--(7b)--(4b)--(3b)--(1b);
\draw (7b)--(9r);
\draw (9r)--(5b);
\draw (5b)--(8r);
\draw (3b)--(7r);
\draw (7r)--(5r);
\draw (5r)--(2b);
\end{tikzpicture}\]
\caption{A non-ambiguous tree and its left and right subtrees}\label{fig_exple_nat}
\end{center}
\end{figure}
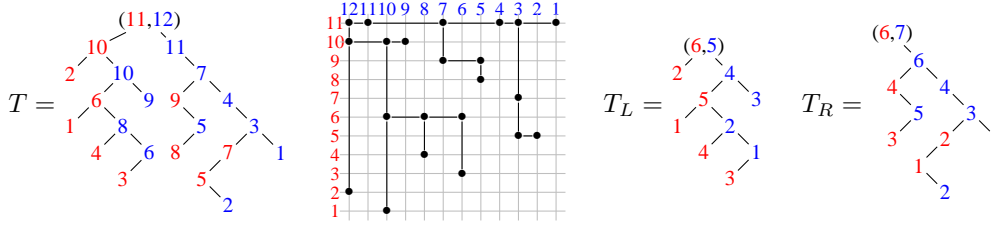

\subsection{Differential equations on non-ambiguous trees}
\label{differential} 

The goal of this section is to get (new) formulas for the number of NATs 
with prescribed shape.
The crucial argument is the following remark: Let $T$ be a NAT
of shape a non empty binary tree $B = \NodeBT{L}{R}$.
Restricting the labellings of the left and right children of $T$
to $L$ and $R$ gives non-decreasing labelling of their respective left and
right children. Note that the root of $L$ (resp. $R$) is a left
(resp. right) child in $T$. By renumbering the labels so that they are
consecutive numbers starting from $1$, we get two non-ambiguous labellings for
$L$ and $R$, that is two non-ambiguous trees $T_L$ and
$T_R$. See Figure~\ref{fig_exple_nat} for an example.

Conversely, knowing the labelling of $L$ and $R$, to recover the labelling
of $T$, one has to choose which labels among $1\dots\LV(T)$ will be used for
$L$ (including its root) and the same for right labels.  
As a consequence:
\begin{equation}\label{Equation:BNat-binom}
  \left|\NATB[\NodeBT{L}{R}]\right| =
  \binom{|\LV(T)|}{|\LV(R)|}
  \binom{|\RV(T)|}{|\RV(L)|}\,
  |\NATB[L]|\,
  |\NATB[R]|.
\end{equation}
Our first step is to recover hook-length formula for the number of
NATs of fixed shape (\cite{AvaBouBouSil14}). We
use the method from~\cite{HivNovThib08}, namely, applying recursively a
bilinear integro-differential operator called here a \emph{pumping function}
along a binary tree.

First of all, we consider the space $\QQ\NAT$ of formal sums of non-ambiguous
trees and identifies $\NATB$ with the formal sum of its elements.
We consider the map $\BNAT:\NAT\times\NAT\mapsto\QQ\NAT$
sending $(T_1,T_2)$ to the formal sum of NATs $T$ such that
$T_L=T_1$ and $T_R=T_2$.
By linearity, we extend $\BNAT$ to a bilinear map
$\QQ\NAT\times\QQ\NAT\mapsto\QQ\NAT$. The main remark is that $\NATB$ can be
computed by a simple recursion using $\BNAT$:
\begin{lem}
  The set $\NATB[B]$ of non-ambiguous tree of shape $B$ satisfies the following recursion:
  \begin{equation}
    \NATB[\emptyset] = \emptyset
    \qquad\text{and}\qquad
    \NATB[\NodeBT{L}{R}] = \BNAT\left(\NATB[L], \NATB[R]\right)\,.
  \end{equation}
\end{lem}
To count non-ambiguous trees, and as suggested by the binomial coefficients in~\eqref{Equation:BNat-binom},
we shall use \emph{doubly exponential generating
  functions} in two variables $x$ and $y$ where $x$ and $y$ count the size of
the rectangle 
in which the NAT is embedded: the weight of the NAT $T$ is
$\Phi(T) := \frac{x^{\wi_L(T)}}{\wi_L(T)!} \frac{x^{\wi_R(T)}}{\wi_R(T)!}\,.$
We extend $\Phi(T)$ by linearity to a map $\QQ\NAT\mapsto\QQ[[x,
y]]$. Consequently, $\Phi(\NATB[B])$ is the generating series of the
non-ambiguous trees of shape $B$. Thanks to~\eqref{Equation:BNat-binom} the image in $\QQ[[x, y]]$ of the bilinear
map $\BNAT$ under the map $\Phi$ is a simple differential operator:
\begin{defi}
  The \emph{pumping function} $\Bxy$ is the bilinear map
  $\QQ[[x,y]]\times\QQ[[x,y]]\mapsto\QQ[[x,y]]$ defined by
  \begin{equation}
    \Bxy(u, v) = \int_x\int_y
    \partial_x(u)\cdot
    \partial_y(v).
  \end{equation}
  We further define recursively, for any binary tree $B$
  an element $\Bxy(B)\in\QQ[[x,y]]$ by
  \begin{equation}
  \Bxy(\emptyset) = x + y
  \qquad\text{and}\qquad
  \Bxy\left({\NodeBT{L}{R}}\right) = \Bxy\left(\Bxy(L), \Bxy(R)\right)\,.
\end{equation}
\end{defi}
Now~\eqref{Equation:BNat-binom} rewrites as
\begin{prop}
\label{prop:commutation_phi}
  For $T_1, T_2\in\QQ\NAT$, one as
  $\Phi(\BNAT(T_1, T_2)) = \Bxy(\Phi(T_1),\Phi(T_2))$.
  As a consequence, for any non empty binary tree $B$,
  $\Phi(\NATB[B])=\Bxy(B)$.
\end{prop}
By de-recursiving the expression for $\Bxy(B)$, we recover the hook-length
formula of~\cite{AvaBouBouSil14} for non-ambiguous trees of a given shape:
\begin{prop}
\label{prop_hook_nat}
Let $\B$ be a binary tree. For each left (resp. right) vertex $U$, we denote
$\EL(U)$ (resp. $\ER(U)$) the number of left (resp. right) vertices of the
subtree with root $U$ (itself included in the count). Then
\begin{equation}
|\NAT(\B)|
=
\frac{ |\LV(\B)|! \cdot |\RV(\B)|!}{
  \displaystyle
  \prod_{U: \text{left child}}\EL(U) \cdot 
  \prod_{U: \text{right child}}\ER(U)
}\,.
\end{equation}
\end{prop}
\bigskip

We consider now the \emph{exponential generating function of non-ambiguous
  trees} with weight $\Phi$:
\begin{equation}
\GFH := \sum_{\N \in \NAT} \Phi(\N) = 
  \sum_{\N \in \NAT} \frac{x^{\wi_L(T)}}{\wi_L(T)!} \frac{x^{\wi_R(T)}}{\wi_R(T)!}\,.
\end{equation}
It turns out that we need to consider the two
following slight modifications to get nice algebraic properties
(because of the empty NAT).
\begin{equation}
\LGFN = \sum_{\B \in \BT} \Bxy(\B)
\hspace{.5cm}
\text{ and }
\hspace{.5cm}
\GFN
=
\sum_{\N\in\NAT^*}
	\frac{
		x^{|\LV(\N)|} \cdot y^{|\RV(\N)|}
	}{
		|\LV(\N)|! \cdot |\RV(\N)|!
	}
.
\end{equation}
The function $\GFH$, $\GFN$, $\LGFN$ are closely related. Each function is
used in different situation. The first one is the natural definition we want
to give. The second one is convenient from a bijective point of view.
The last one is convenient from the algebraic and analytic point of view.
They differ by their constant term and shift in the degree. Precisely,
$\GFN = \partial_x \partial_y \GFH$ so that
\begin{equation}
\GFH = 1 + \int_x \int_y \GFN
\hspace{.5cm}
\text{ and }
\hspace{.5cm}
\LGFN = x + y + \int_x \int_y \GFN
\hspace{.5cm}
\text{ and }
\hspace{.5cm}
\LGFN = \GFH + x + y - 1
\end{equation}
The two last relations are consequences of Proposition~\ref{prop:commutation_phi}.
\begin{prop}
\label{prop_equ_diff_nat}
The generating function $\GFN$ and $\LGFN$ can be computed by the following
fixed point differential equations:
\begin{equation}
\label{equ_gfn}
\LGFN = x + y + \int_x \int_y \partial_x \LGFN \cdot \partial_y \LGFN 
\hspace{.5cm}
\text{and}
\hspace{.5cm}
\GFN = 
\left( 1 + \int_{x} \GFN  \right)
\cdot
\left( 1 + \int_{y} \GFN  \right)
\end{equation}
\end{prop}

\begin{proof}
The first equation is a just a consequence of the definition of the bilinear map $\Bxy$:
$$
\LGFN
=
x + y + \sum_{L, R \in \BT} 
\Bxy\left({\NodeBT{L}{R}}\right) 
=
x + y + \sum_{L, R \in \BT} \Bxy(\Bxy(L), \Bxy(R))
=
x + y + \Bxy( \LGFN, \LGFN )
.
$$
To prove the second equation, remark that the first can be rewritten as
$\partial_x \partial_y \LGFN = \partial_x \LGFN. \partial_y \LGFN$. So that,
$\GFN = \partial_x \partial_y \GFH = \partial_x \partial_y \LGFN$.
To conclude, it suffices to remark that
$\partial_x \LGFN = 1 + \int_y \GFN$
\end{proof}

Now, a closed formula can be computed for $\GFN$ and $\GFH$.

\begin{prop}
The exponential generating function for non-ambiguous trees are given by
$$
\GFN
=
\frac{
e^{x+y}
}{
\left(
1 - (e^x - 1)(e^y - 1)
\right)^2
}\,,
\quad\text{and}\quad
\GFH = -\log( 1 - (e^x-1)(e^y-1) ).
$$
\end{prop}


Now, we will introduce two statistics : the number of right (resp. left) 
vertices in the rightmost (resp. leftmost) branch of the root of a tree.
For a binary tree $\B$, we will denote by $\RO(\B)$ (resp. $\LO(\B)$) the two
previous statistics.
We define now an $(\alpha, \beta)$-generating function for non-ambiguous trees:
$$
\GFN^{(\alpha, \beta)}
=
\sum_{\N \in \NAT}
	\frac{
		x^{|\LV(\N)|} \cdot y^{|\RV(\N)|} \cdot \alpha^{\RO(\N)} \cdot \beta^{\LO(\N)}
	}{
		|\LV(\N)|! \cdot |\RV(\N)|!
	}.
$$
\begin{prop}\label{diff_eq_nat}
A differential equation for $\GFN^{(\alpha, \beta)}$ is
$$
		\GFN^{(\alpha, \beta)}	=
		\left(
			1 + \alpha \int_x \GFN^{(\alpha, 1)}
		\right)
		\cdot
		\left(
			1 + \beta \int_y \GFN^{(1, \beta)}
		\right),
$$
\end{prop}
\begin{proof}
We just have to define a new pumping function by setting 
$\Bxy^{(\alpha, \beta)}(\B) = \alpha^{\RO(\B)} \beta^{\LO(\B)} \Bxy(\B)$
and deduce the expected differential equation.
\end{proof}

The solution of the new differential equation is given by 
Proposition~\ref{pro_equ_diff_alpha_beta}.
\begin{prop}\label{gen_ser_nat}
The $(\alpha,\beta)$-exponential generating function for non-ambiguous trees is equal to
\label{pro_equ_diff_alpha_beta}
$$
\GFN^{(\alpha, \beta)}
=
\frac{
	e^{\alpha x + \beta y}
}{
	\left(
		1 - (e^x -1)(e^y - 1)
	\right)^{\alpha+\beta}
}.
$$
\end{prop}

\subsection{Bijection with some labelled ordered trees} \label{bijectionslabelledtree} 

In what follows, we will use \emph{rooted ordered trees}. These are trees such that each node has an ordered (possibly
empty) list of children. We draw the children from left to right on the pictures.

Note that the solution of Proposition \ref{pro_equ_diff_alpha_beta} can be rewritten as : 

\begin{equation}\label{eq:gfnab}
\GFN^{(\alpha, \beta)} = e^{\alpha x} e^{\beta y} e^{-\alpha\ln(1-(e^x -1)(e^y - 1))} e^{-\beta \ln(1-(e^x -1)(e^y - 1))}.
\end{equation}
The purpose of this subsection is to explain this expression combinatorially. 
Let us first describe objects ``naturally'' enumerated by the RHS of \eqref{eq:gfnab}.
 We recall that $e^x$ is the exponential 
generating series of sets and  $-\ln(1-x)$ is the exponential generating series 
of cycles. The objects can be described as $4$-tuples consisting of two sets of elements and two sets of cycles whose elements are pairs of non empty sets. Let us denote by $\ft$ the set of such $4$-tuples.

We first link non-ambiguous trees with ordered trees. We need the following 
definition:

\begin{defi} \label{def_hook}
Let $B$ be a binary tree and $v$ one of his node. The \emph{hook} of a vertex
 $v$ is the union of $\{v\}$, its leftmost branch and its rightmost branch. There is a unique way to partition the vertices in hooks. The number of hooks in such a partition is the \emph{hook number of the tree}.
\end{defi}

\begin{rque}
We can obtain recursively the unique partition of the preceding definition by 
extracting the root's hook and iterating the process on each tree of the remaining forest.
\end{rque}

\begin{exple} On the left part of Figure \ref{fig_nat}, we represented in red
  the hook of \textcolor{blue}{10}. The partition of vertices in hooks is obtained by
  removing the dotted edges. The hook number of the tree is $8$.
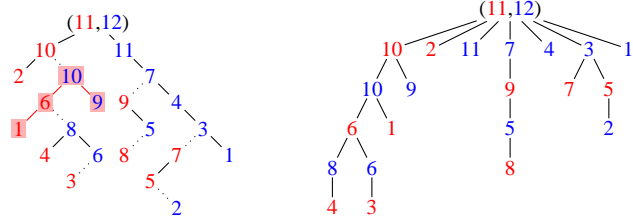
\begin{figure} 
\begin{center}
\begin{tikzpicture}[scale=0.35]
\node (r) at (0,8) {(\textcolor{red}{11},\textcolor{blue}{12})};
\node (11b) at (1,7) {\textcolor{blue}{11}};
\node[fill=red!30] (10b) at (-1,6) {\textcolor{blue}{10}};
\node[fill=red!30] (9b) at (0,5) {\textcolor{blue}{9}};
\node (8b) at (-1,4) {\textcolor{blue}{8}};
\node (7b) at (2,6) {\textcolor{blue}{7}};
\node (6b) at (0,3) {\textcolor{blue}{6}};
\node (5b) at (2,4){\textcolor{blue}{5}};
\node (4b) at (3,5) {\textcolor{blue}{4}};
\node (3b) at (4,4){\textcolor{blue}{3}};
\node (2b) at (3,1){\textcolor{blue}{2}};
\node (1b) at (5,3){\textcolor{blue}{1}};
\node (10r) at (-2,7){\textcolor{red}{10}};
\node (9r) at (1,5){\textcolor{red}{9}};
\node (8r) at (1,3){\textcolor{red}{8}};
\node (7r) at (3,3){\textcolor{red}{7}};
\node[fill=red!30] (6r) at (-2,5){\textcolor{red}{6}};
\node (5r) at (2,2){\textcolor{red}{5}};
\node (4r) at (-2,3){\textcolor{red}{4}};
\node (3r) at (-1,2){\textcolor{red}{3}};
\node (2r) at (-3,6){\textcolor{red}{2}};
\node[fill=red!30] (1r) at (-3,4){\textcolor{red}{1}};
\draw (2r)--(10r)--(r)--(11b)--(7b)--(4b)--(3b)--(1b);
\draw [dotted] (7b)--(9r);
\draw (9r)--(5b);
\draw [dotted] (5b)--(8r);
\draw [dotted] (3b)--(7r);
\draw (7r)--(5r);
\draw [dotted] (5r)--(2b);
\draw [dotted] (10r)--(10b);
\draw[red] (9b)--(10b)--(6r)--(1r);
\draw (6b)--(8b)--(4r);
\draw [dotted] (3r)--(6b);
\draw [dotted] (6r)--(8b);
\end{tikzpicture}
\hspace{1cm}
\begin{tikzpicture}[scale=0.35,sibling distance=1.5cm]
\node{(\textcolor{red}{11},\textcolor{blue}{12})}  
   child{node{\textcolor{red}{10}}
   		child{node{\textcolor{blue}{10}}
   				child{node{\textcolor{red}{6}}
   				    child{node{\textcolor{blue}{8}}
   				        child{node{\textcolor{red}{4}}}}
   				    child{node{\textcolor{blue}{6}}
   				    child{node{\textcolor{red}{3}}}}
   				    }
   				child{node{\textcolor{red}{1}}}}
   	    child{node{\textcolor{blue}{9}}}}
   child{node{\textcolor{red}{2}}}
   child{node{\textcolor{blue}{11}}}
   child{node{\textcolor{blue}{7}}
       child{node{\textcolor{red}{9}}
           child{node{\textcolor{blue}{5}}
               child{node{\textcolor{red}{8}}}}}}
   child{node{\textcolor{blue}{4}}}
   child{node{\textcolor{blue}{3}}
child{node{\textcolor{red}{7}}}
       child{node{\textcolor{red}{5}}[scale=0.75]
           child{node{\textcolor{blue}{2}}}}}
   child{node{\textcolor{blue}{1}}}
   	;
\end{tikzpicture}
\caption{Hooks on a non-ambiguous tree and associated ordered tree}\label{fig_nat}
\end{center}
\end{figure}
\end{exple}

We denote by $\NOT$ 
the set of ordered trees $O$ such that:
\begin{itemize}
\item Each vertex, except the root, is labelled and coloured (in red or blue). The root is labelled by a red label and a blue label, both maximal.
\item The root has red and blue children, the red children being on the left side of blue children. Blue (resp. red) vertices have only red (resp. blue) children.
\item The labels of red (resp. blue) descendants or right siblings of a red 
(resp. blue) vertex $v$ are smaller than the label of $v$.
\end{itemize}
\begin{prop} \label{bij_pair_ordered}
The set of non-ambiguous trees $\NAT$ on $n$ nodes is in bijection with the 
set of trees of $\NOT$ on $n$ nodes. This bijection is denoted by $\xi$.
\end{prop}

\begin{proof}
Let us consider a non-ambiguous tree $\N$ and construct an ordered tree 
$\xi(\N)=\O$. The root of $\N$ will be associated to the root of $\O$. 
Starting from the root $r$ of the ordered tree, the red (resp. blue) 
children of $r$ are the set of left (resp. right) descendants of the root 
of $\N$. The expected ordered tree is then obtained recursively by the 
following rule : if a node $v$ in the ordered tree is a left (resp. right) 
child in $\N$, then its children in the ordered tree is the set of right 
(resp. left) descendants of $v$ in $\N$, with every right (resp. left) child
on the right side of its parent. 

We can reconstruct recursively the non-ambiguous tree associated to such 
an ordered tree, by reversing the process from the children of the root to 
the leaves in the ordered tree.
\end{proof}

\begin{rque}
Let us remark that the hook of a vertex $v$, different from the root in 
the non-ambiguous tree, can be read off from the ordered trees : it consists 
in the children of $v$ in the ordered tree and the siblings of $v$ on 
the right side of $v$ in the ordered tree.
\end{rque}

\begin{exple}
The ordered tree associated to the non-ambiguous tree on the left part of Figure \ref{fig_nat} is represented on the right part of the same figure.

\begin{figure}
\begin{center}

\end{center}
\end{figure}
\end{exple}

\begin{prop}
The set of non-ambiguous trees $\NAT$ is in bijection with pairs of 2-coloured words, 
with blue letters on $\{1, \ldots, |\RV|\}$ and red letters on $\{1, \ldots, |\LV|\}$,
where each letter appear exactly once (in the first word or in the second word),
letters in blocks of the same colors are decreasing,
the first (resp. second) word ends by a red (resp. blue) letter
and $\RV$ (resp. $\LV$) is the set of right (resp. left) children in the non-ambiguous tree.
This bijection is denoted by $\xi \circ \Omega$.
Moreover, the pairs of 2-coloured words are exactly described by the previous $4$-tuples.
\end{prop}

\begin{proof}[(sketch)] From $T \in \NOT$, we obtain the two words $\Omega(T)=(w_1,w_2)$ by a post-order traversal visit of the descendant of the red (resp. blue) children of the root for $w_1$ (resp. $w_2$). The injectivity of $\Omega$ can be shown \textit{ad absurdum}.

From such a word, we can build back recursively the associated ordered trees by reading each word from right to left and adding, for each new letter $l$, a node labelled by $l$ to the left of the closest ancestor of the current position whose label is of the same colour as $l$ and  smaller than $l$.

The consecutive maximal red (rep. blue) elements from right to left in the first (resp. second) word correspond to the children of the root in the ordered tree. The first (resp. second) set of the $4$-tuple can be defined as the set of blue (resp. red) children of the root in the ordered tree. Then, each remaining subword, corresponding to one child of the root and its descendants in the ordered tree, contains both blue and red elements, the rightmost letter corresponding to the child of the root. Each of these subwords can be viewed as a blue (resp. red) cycle, as the child of the root is the biggest blue (resp. red) element in the subword and can be found again. This cycle is made of alternating sets of blue and red elements, corresponding to right and left vertices in the non-ambiguous tree, which can be joined in pairs of non empty sets, giving the two set of cycles of the $4$-tuple.
\end{proof}

\begin{exple}
The pair of words associated with trees of Figure \ref{fig_nat} is (\textcolor{red}{4}\,\textcolor{blue}{8}\ \textcolor{red}{3}\,\textcolor{blue}{6}\,\textcolor{red}{6}\,\textcolor{red}{1}\,\textcolor{blue}{10}\,\textcolor{blue}{9}\,\textcolor{red}{10}\,\textcolor{red}{2}, \textcolor{blue}{11}\,\textcolor{red}{8}\,\textcolor{blue}{5}\,\textcolor{red}{9}\,\textcolor{blue}{7}\,\textcolor{blue}{4}\,\textcolor{red}{7}\ \textcolor{blue}{2}\,\textcolor{red}{5}\,\textcolor{blue}{3}\,\textcolor{blue}{1}). The associated $4$-tuple is: (\{\textcolor{red}{2}\}, \{\textcolor{blue}{1},\textcolor{blue}{4},\textcolor{blue}{11}\}, \{\textcolor{red}{(}\{\textcolor{red}{10}\ \textcolor{red}{4}\,\textcolor{blue}{8}\}\{\textcolor{red}{3}\textcolor{blue}{6}\}\{\textcolor{red}{6}\textcolor{red}{1}\textcolor{blue}{10}\,\textcolor{blue}{9}\}\textcolor{red}{)}\}, \{\textcolor{blue}{(}\{\textcolor{red}{8}\textcolor{blue}{5}\}\{\textcolor{red}{9}\textcolor{blue}{7}\}\textcolor{blue}{)},\textcolor{blue}
 {(}\{\textcolor{red}{7}\textcolor{blue}{2}\}\{\textcolor{red}{5}\textcolor{blue}{3}\}\textcolor{blue}{)}).
\end{exple}

\begin{rque}
The bijection $\Omega$ is similar to the  ``zigzag'' bijection of \cite{SteiWill}.
\end{rque}

We may derive from our construction a bijective proof of the following enumeration formula.

\begin{thm}\label{thm_p}
The $(\alpha,\beta)$-analogue of the number of non empty non-ambiguous trees with $w$ left vertices and $h$ right vertices is given by:
\begin{equation}
\NAT_{w,h} = \sum_{p \geq 1}
    (p-1)!\cdot(p-1)^{(\alpha+\beta)}\cdot
    S_{2,\alpha}(w+1, p) S_{2,\beta}(h+1, p)
\end{equation}
where $p^{(q)}$ is the rising factorial, and
$S_{2,q}$ denotes the $q$-analogue of the Stirling numbers of the second kind
such that, if we consider a set partition,
$q$ counts the number of elements different from $1$ in the subset containing $1$.
In this positive summation expression, each summand corresponds 
to the number of NATs with prescribed size, and whose number of hooks equals $p$.
\end{thm}

We conclude this subsection with following result on binary trees.
The corresponding integer series appears as \cite[A127157]{oeis} in OEIS.
\begin{prop} The set of binary trees on $n$ vertices with hook number $p$ is in bijection with the number of ordered trees on $n+1$ vertices having $p$ vertices being the parent of at least a leaf.
\end{prop}

\subsection{q-analogs of the hook formula}\label{q_analogs} 

As for binary trees, there exists $q$-analogues of the hook formula for NATs 
of a given shape associated to either the number of inversions
or the major index. There are two ingredients: first we need to associate two
permutations to a non-ambiguous tree, and second we need to give a $q$-analogue
of the bilinear map $\Bxy$. It turns out that it is possible to use two
different $q$ namely $q_R$ and $q_L$ for the derivative and integral in $x$
and $y$.

The first step to formulate a $q$-hook formula is to associate to any
non empty non-ambiguous  tree $T$ a pair of permutations $\sigma(T)=(\sigma_L(T),
\sigma_R(T))\in\SG_{\LV(T)} \times \SG_{\RV(T)}$.
\begin{defi}
  Let $T$ be a non-ambiguous tree. Then $\sigma_L(T)$ is obtained by
  performing a left postfix reading of the left labels: precisely we
  recursively read trees $\NodeBT{L}{R}$ by reading the left labels of $L$,
  then the left labels of $R$ and finally the label of the root if it is a left
  child. The permutation $\sigma_R(T)$ is defined similarly reading right
  labels, starting from the right subtree, then the left subtree and finally
  the root.
\end{defi}
If we take back the example of Figure~\ref{fig_exple_nat} we get the two permutations $\sigma_L(T) = (2, 1, 4, 3, 6, 10, 8, 9, 5, 7)$ and $\sigma_R(T) = (1, 2, 3, 4, 5, 7, 11, 9, 6, 8, 10)$.

Recall that the \emph{number of inversions} of a permutation $\sigma\in\SG_n$
is the number of $i<j<=n$ such that $\sigma(i)>\sigma(j)$. A descent of
$\sigma$ is a $i<n$ such that $\sigma(i)>\sigma(i+1)$ and the \emph{inverse
  major index} of $\sigma$ is the sum of the descents of $\sigma^{-1}$. Finally
for a repetition free word $w$ of length $l$ we write $\std(w)$ the
permutations in $\SG_l$ obtained by renumbering $w$ keeping the order of the
letters. For example $\std(36482)=24351$. We define as usual the $q$-integer
$[n]_q := \frac{1-q^n}{1-q}$, and the $q$-factorial $[n]_q! := \prod_{i=1}^{n}
[i]_q$.
\begin{thm}\label{thm-q-hook}
For a non-ambiguous tree $\N$ and a statistic $S\in\{\inv,\imaj\}$, define
\begin{equation}
  w_S(T) := q_L^{S(\sigma_L(T))}q_R^{S(\sigma_R(T))}.
\end{equation}
Then, for any non empty binary tree $B$
\begin{equation}
\sum_{T\in\NAT(\B)} w_{\inv}(T) = 
\sum_{T\in\NAT(\B)} w_{\imaj}(T) =
\frac{ |\LV(\B)|_{q_L}! \cdot |\RV(\B)|_{q_R}!}{
  \displaystyle\prod_{U: \text{left child}}[\EL(U)]_{q_L} \cdot 
  \prod_{U: \text{right child}}[\ER(U)]_{q_R}
}\,.
\end{equation}
\end{thm}
Going back to the non-ambiguous tree of Figure~\ref{fig_exple_nat}, the
inversions numbers are $\inv(\sigma_L(T)) = 11$ and, $\inv(\sigma_R(T)) = 7$ so
that $w_{\inv}(T) = q_L^{11}q_R^{7}$. For the inverse major index, we get the permutations $\sigma_L(T)^{-1} = (2, 1, 4, 3, 9, 5, 10, 7, 8, 6)$ and $
\sigma_R(T)^{-1} = (1, 2, 3, 4, 5, 9, 6, 10, 8, 11, 7)$.
Consequently, $\imaj(\sigma_L(T)) = 1 + 3 + 5 + 7 + 9 = 25$
and     $\imaj(\sigma_R(T)) = 6 + 8 + 10 = 24$ so that 
$w_{\imaj}(T) = q_L^{25}q_R^{24}$.

Note that it is possible to read directly $w_S(T)$ on $T$.  We do not give the
precise statement here to keep the presentation short.

The argument of the proof follows the same path as for the hook formula, using
pumping functions: recall that the $q$-derivative and $q$-integral are
defined as $\partial_{x, q} x^n := [n]_qx^{n-1}$ and
$\int_{x, q} x^n := \frac{x^{n+1}}{[n+1]_q}$.
Then the $(q_L, q_R)$-analogue of the pumping function is given by
\begin{equation}
  \Bxy_q(u, v) = \int_{x,q_L}\int_{y,q_R}
  \partial_{x,q_L}(u)\cdot
  \partial_{y,q_R}(v).
\end{equation}
We also define recursively $\Bxy_q(B)$ by $\Bxy_q(\emptyset) := x + y$ and
$\Bxy_{q}\left(\NodeBT{L}{R}\right) = \Bxy_q\left(\Bxy_q(L), \Bxy_q(R)\right)\,$. Then
the main idea is to go through a pumping function on pairs of permutations. We
write $\QQ\SG$ the vector space of formal sums of permutations. For any
permutation $\sigma\in\SG_n$ we write $\int\sigma=\sigma[n+1]$ the
permutation in $\SG_{n+1}$ obtained by adding $n+1$ at the end. Again we
extend $\int$ by linearity.
\begin{defi}
  The \emph{pumping function on permutation} is the bilinear map
  $\BSG:\QQ\SG\times\QQ\SG\mapsto\QQ\SG$ defined
  for $\sigma\in\SG_m$ and $\mu\in\SG_n$ by $
    \BSG(\sigma,\mu) =
    \sum_{\substack{uv\in\SG_{m+n+1} \\ \std(u)=\int\sigma \\ \std(v)=\mu}} uv\,.$
    
  We define also a pumping function on pairs of permutations
  \begin{equation*}
    \BSG^2\left((\sigma_L,\sigma_R),(\mu_L,\mu_R)\right) :=
    (\BSG(\sigma_L,\mu_L),\BSG(\mu_R,\sigma_R))
  \end{equation*}
\end{defi}
For example
$\BSG(21,12)=21345+21435+21534+31425+31524+
41523+32415+32514+42513+43512$.
Note that for two non empty non-ambiguous tree $C,D$.
\begin{equation*}
  \sum_{T\in\BNAT(C, D)} \sigma_L(T) =
  \BSG(\sigma_L(C),\sigma_L(D))
  \qquad\text{and}\qquad
  \sum_{T\in\BNAT(C, D)} \sigma_R(T) =
  \BSG(\sigma_R(D),\sigma_R(C))
\end{equation*}
The central argument is the following commutation property:
\begin{prop}
  For a statistic $S\in\{\inv,\imaj\}$,
  and $(\sigma_L,\sigma_R)\in\SG_m\times\SG_n$, define
  \begin{equation}
  \Psi_S((\sigma_L,\sigma_R)) :=
  q_L^{S(\sigma_L)}\frac{x^{m+1}}{[m+1]_{q_L}!}\,
  q_R^{S(\sigma_R)} \frac{y^{n+1}}{[n+1]_{q_L}!}\,.
\end{equation}
Then for any pairs $\sigma=(\sigma_L,\sigma_R)$ and $\mu=(\mu_L,\mu_R)$, one
has $\Psi_S(\BSG^2(\sigma, \mu)) = \Bxy_q(\Psi_S(\sigma), \Psi_S(\mu))$
\end{prop}
As a consequence, noting that $w_S(T) = \Phi_S(\sigma(T))$, one finds that for
any non empty non-ambiguous trees $C$ and $D$, 
\begin{equation*}
  \sum_{T\in\BNAT(C, D)} w_S(T) = 
  \Phi_S\left(\BSG^2(\sigma(C),\sigma(D)\right)
  = \Bxy_q(w_S(C), w_S(D))\,.
\end{equation*}
Applying this recursively on the structure of a binary tree $B$, we have that
$ \sum_{T\in\NATB} w_S(T) = \Bxy_q(B)\,.$ Unfolding the recursion for
$\Bxy_q(B)$, gives finally Theorem~\ref{thm-q-hook}.

We conclude this section by an example. Let $B = \scalebox{0.5}
{ \newcommand{\nodea}{\node[draw,fill,circle] (a) {$$};}
  \newcommand{\nodeb}{\node[draw,fill,circle] (b) {$$};}
  \newcommand{\nodec}{\node[draw,fill,circle] (c) {$$};}
  \newcommand{\noded}{\node[draw,fill,circle] (d) {$$};}
  \newcommand{\nodee}{\node[draw,fill,circle] (e) {$$};}
  \newcommand{\nodef}{\node[draw,fill,circle] (f) {$$};}
  \newcommand{\nodeg}{\node[draw,fill,circle] (g) {$$};}
  \newcommand{\nodeh}{\node[draw,fill,circle] (h) {$$};}
\begin{tikzpicture}[scale=0.5,baseline=(current bounding box.center),
                      every node/.style={inner sep=2pt}]
\matrix[column sep=.15cm, row sep=.15cm,ampersand replacement=\&]{
         \&         \&         \& \nodea  \&         \&         \&         \&         \&         \\
         \& \nodeb  \&         \&         \&         \&         \&         \& \nodee  \&         \\
 \nodec  \&         \& \noded  \&         \&         \& \nodef  \&         \&         \& \nodeh  \\
         \&         \&         \&         \&         \&         \& \nodeg  \&         \&         \\
};
\path[thick] (b) edge (c) edge (d)
	(f) edge (g)
	(e) edge (f)
	(a) edge (b) edge (e)
        (e) edge (h);
\end{tikzpicture}}$.
Then one finds that the $q$- hook formula gives $(qx^3 + qx^2 + qx + 1)(qy^2 +
qy + 1)(qx + 1)$. Expanding this expression, one finds that the coefficient of $qx^2qy$ is $2$.
For the $\imaj$ statistic it corresponds to the two following non-ambiguous
trees which are shown with their associated left and right permutations:
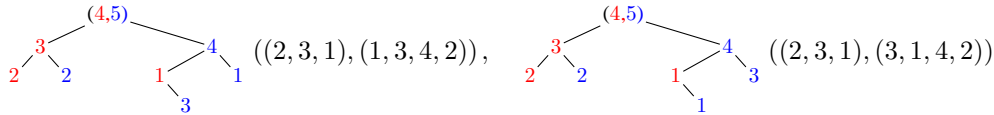
\[
{\newcommand{\nodea}{\node (a) {(\color{red}$4$,\color{blue}$5$)}
;}\newcommand{\nodeb}{\node (b) {\color{red}$3$}
;}\newcommand{\nodec}{\node (c) {\color{red}$2$}
;}\newcommand{\noded}{\node (d) {\color{blue}$2$}
;}\newcommand{\nodee}{\node (e) {\color{blue}$4$}
;}\newcommand{\nodef}{\node (f) {\color{red}$1$}
;}\newcommand{\nodeg}{\node (g) {\color{blue}$3$}
;}\newcommand{\nodeh}{\node (h) {\color{blue}$1$}
;}\begin{tikzpicture}[baseline=(current bounding box.center)]
\matrix[column sep=1.5mm, row sep=1.5mm,ampersand replacement=\&]{
         \&         \&         \& \nodea  \&         \&         \&         \&         \&         \\ 
         \& \nodeb  \&         \&         \&         \&         \&         \& \nodee  \&         \\ 
 \nodec  \&         \& \noded  \&         \&         \& \nodef  \&         \&         \& \nodeh  \\ 
         \&         \&         \&         \&         \&         \& \nodeg  \&         \&         \\
};

\path (b) edge (c) edge (d)
	(f) edge (g)
	(e) edge (f) edge (h)
	(a) edge (b) edge (e);
\end{tikzpicture}}
\left((2, 3, 1), (1, 3, 4, 2)\right),\quad
{ \newcommand{\nodea}{\node (a) {(\color{red}$4$,\color{blue}$5$)}
;}\newcommand{\nodeb}{\node (b) {\color{red}$3$}
;}\newcommand{\nodec}{\node (c) {\color{red}$2$}
;}\newcommand{\noded}{\node (d) {\color{blue}$2$}
;}\newcommand{\nodee}{\node (e) {\color{blue}$4$}
;}\newcommand{\nodef}{\node (f) {\color{red}$1$}
;}\newcommand{\nodeg}{\node (g) {\color{blue}$1$}
;}\newcommand{\nodeh}{\node (h) {\color{blue}$3$}
;}\begin{tikzpicture}[baseline=(current bounding box.center)]
\matrix[column sep=1.5mm, row sep=1.5mm,ampersand replacement=\&]{
         \&         \&         \& \nodea  \&         \&         \&         \&         \&         \\ 
         \& \nodeb  \&         \&         \&         \&         \&         \& \nodee  \&         \\ 
 \nodec  \&         \& \noded  \&         \&         \& \nodef  \&         \&         \& \nodeh  \\ 
         \&         \&         \&         \&         \&         \& \nodeg  \&         \&         \\
};

\path (b) edge (c) edge (d)
	(f) edge (g)
	(e) edge (f) edge (h)
	(a) edge (b) edge (e);
\end{tikzpicture}}
\left((2, 3, 1), (3, 1, 4, 2)\right)\]

\section{Non-ambiguous trees in higher dimension}
\label{gnat}

In this section we give a generalisation of NATs to higher dimensions.
NATs are defined as binary trees whose vertices are 
embedded in $\mathbb{Z}^2$, and edges are objects of dimension 1 (segments).
Let $d \ge k \ge 1$ be two integers.
In higher dimension, binary trees are replaced by $\binom{d}{k}$-ary trees 
embedded in $\mathbb{Z}^d$ 
and edges are objects of dimension $k$.
As in Section~\ref{differential} we obtain differential equations for these objects.

\subsection{Definitions}
\label{def_natdk} 

We call \emph{$(d,k)$-direction} a subset of cardinality $k$ of $\{1, \ldots, d\}$.
The set of $(d,k)$-directions is denoted by $\edir$. 
A $(d,k)$-tuple is a $d$-tuple of $(\mathbb{N}\cup\{\bullet\})^d$, 
in which $k$ entries are integers and $d-k$  are $\bullet$.
For instance,
$
(\bullet, \ 1, \  \bullet, \  5,\  2, \ \bullet, \ \bullet, \ 3, \ \bullet)
$
is a $(9,4)$-tuple.
The direction of a $(d,k)$-tuple $U$ is the set indices of $U$
corresponding to entries different from $\bullet$.
For instance, the direction of our preceding example is $\{2,4,5,8\}$.

\begin{defi} 
A \emph{$\binom{\d}{\k}$-ary} tree $\M$ is a tree whose children of given vertex
are indexed by a $(\d,\k)$-direction.
\end{defi}

A $(\d,\k)$-ary tree has at most $\binom{\d}{\k}$ children.
A $\binom{\d}{\k}$-ary tree will be represented as an ordered tree where the 
children of a vertex $S$ are drawn from left to right with respect to the 
lexicographic order of their indices.
If a vertex $S$ has no child associated to an index $\dir$, we draw an half 
edge in this direction.
An example is drawn on Figure~\ref{fig_gnat}.


\begin{defi}\label{def_gnat}
A \emph{non-ambiguous tree of dimension $(d,k)$} is a labelled 
\emph{$\binom{d}{k}$-ary} tree such that:
\begin{enumerate}
\item\label{gnat_dk_tuple} a child of index $\dir$ is labelled with a $(d,k)$-tuple of direction 
$\dir$ and \label{gnat_root} the root is labelled with a $(d,d)$-tuple;
\item\label{gnat_growth} for any descendant $U$  of $V$, if the $i$-th component of $U$ 
and $V$ are different from $\bullet$, then the $i$-th component of $V$ is
strictly greater than the $i$-th component of $U$;
\item\label{gnat_distinct} for each $i \in \i{1}{\d}$, all the $i$th components,
different from $\bullet$, are pairwise distinct and \label{gnat_interval} the set of $i$th components,
different from $\bullet$, of every vertices in the tree, is an interval,
whose minimum is $1$.
\end{enumerate}

The set of non-ambiguous trees of dimensions $(d,k)$ is denoted by $\NATdk$.
\end{defi}

We write \nat for a non-ambiguous tree (of dimensions $(d,k)$).
Figure \ref{fig_gnat} gives an example of a \nat[3][1] and a \nat[3][2].

\begin{figure}[h]
    \begin{center}
        \scalebox{0.4}{
            \begin{tikzpicture}[
                level 1/.style={sibling distance=7.5cm},
                level 2/.style={sibling distance=2.5cm},
                level 3/.style={sibling distance=2.5cm},
                half/.style={
                    edge from parent/.style=edge
                },
                edge/.code={
                    \draw (\tikzparentnode) -- ($(\tikzparentnode)!6mm!(\tikzchildnode)$);
                }
            ]
                \node{\Huge(5,7,6)}
                    child{ node{\Huge(4,$\bullet$,$\bullet$)}
                        child{ node{\Huge(1,$\bullet$,$\bullet$)} }
                        child[half]
                        child{ node{\Huge($\bullet$,$\bullet$,5)}
                            child[half]
                            child{ node{\Huge($\bullet$,5,$\bullet$)}
                                child[half]
                                child{ node{\Huge($\bullet$,3,$\bullet$)} }
                                child{ node{\Huge($\bullet$,$\bullet$,4)} }
                            }
                            child{ node{\Huge($\bullet$,$\bullet$,2)} }
                        }
                    }
                    child{ node{\Huge($\bullet$,4,$\bullet$)}
                        child[half]
                        child[half]
                        child{ node{\Huge($\bullet$,$\bullet$,1)} }
                    }
                    child{ node{\Huge($\bullet$,$\bullet$,4)}
                        child{ node{\Huge(2,$\bullet$,$\bullet$)} }
                        child{ node{\Huge($\bullet$,6,$\bullet$)} 
                            child{ node{\Huge(3,$\bullet$,$\bullet$)}
                                child[half]
                                child{ node{\Huge($\bullet$,2,$\bullet$)} }
                                child[half]
                            }
                            child{ node{\Huge($\bullet$,1,$\bullet$)} }
                            child[half]
                        }
                        child[half]
                    }
                    ;

            \end{tikzpicture}
        }
        \hspace{1cm}
        \scalebox{0.6}{
            \begin{tikzpicture}[
                level 1/.style={sibling distance=2cm},
                level 2/.style={sibling distance=1.5cm},
                half/.style={
                    edge from parent/.style=edge
                },
                edge/.code={
                    \draw (\tikzparentnode) -- ($(\tikzparentnode)!6mm!(\tikzchildnode)$);
                }
            ]
                \node{\LARGE(6,5,4)}
                    child{ node{\LARGE(5,3,$\bullet$)}
                        child{ node{\LARGE(3,1,$\bullet$)} }
                        child{ node{\LARGE(2,$\bullet$,2)} }
                        child[half]
                    }
                    child{ node{\LARGE(1,$\bullet$,1)} }
                    child{ node{\LARGE($\bullet$,4,3)}
                        child{ node{\LARGE(4,2,$\bullet$)} }
                        child[half]
                        child[half]
                    }
                ;
            \end{tikzpicture}
        }
    \end{center}
    \caption{A NAT of dimension $(3,1)$ and a NAT of dimension $(3,2)$.}
    \label{fig_gnat}\label{ex_kdtree}
\end{figure}
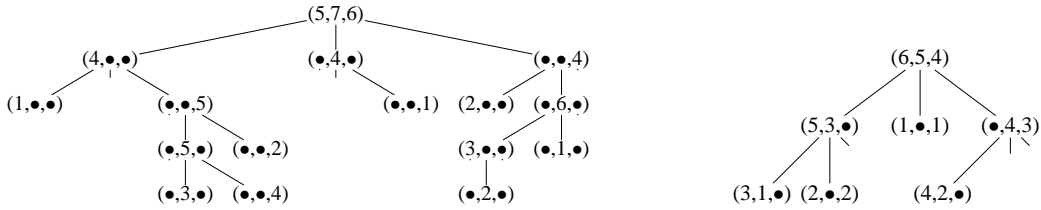

\begin{defi}\label{gnat_gs}

The \emph{geometric size} of a \nat is the $d$-tuple of integers 
$(w_1, \ldots, w_d)$ which labels the root of the \nat, it is denoted by $\gs$.
The \emph{$\dir$-size} of a \nat   is the number of vertices in the tree of
direction $\dir$, the set of such vertices is denoted by $\DV$. 
\end{defi}

Proposition~\ref{prop_relation_sizes} gives the relation between the 
geometric size and the $\dir$-size of a non-ambiguous trees.

\begin{prop}
\label{prop_relation_sizes}
Let $\M$ be a $\binom{d}{k}$-ary tree, the root label is constant on 
$\NATdk$s of shape $M$ ($\NATdk(\M)$):
$$
\wi_i(M)
:=
\wi_i
=
\sum_{\dir \in \edir \ \mid\ i \in \dir} |\DV(\M)| +1
.
$$
\end{prop}

\subsection{Associated differential equations} 

In this section, we denote by $x_{\{i_1, \ldots, i_k\}}$ the product 
$x_{i_1} \times \ldots \times x_{i_k}$,
by $\partial_{\{i_1, \ldots, i_k\}}$ the operator
$\partial_{x_{i_1}} \partial_{x_{i_2}} \ldots \partial_{x_{i_k}}$
and by $\int_{\{i_1, \ldots, i_k\}}$ the operator
$\int_{x_{i_1}} \int_{x_{i_2}} \ldots \int_{x_{i_k}}$.
As for non-ambiguous trees (Proposition \ref{prop_hook_nat}),
there is a hook formula  for the number of non-ambiguous trees
with fixed underlying tree.
Let $\M$ be a $\binom{d}{k}$-ary tree, for each vertex $U$
we denote by $\E_i(U)$ the number of vertices, of the subtree whose root is $U$
(itself included in the count), whose direction contains $i$.


\begin{equation}
|\NATdk(M)|
=
\prod_{i=1}^d 
\left( \wi_i( \M ) - 1 \right) !
\left(\displaystyle
		\prod_{
        \substack{
            U: \text{ child of direction containing } i
        }
    }{
        \E_i(U)
		}
	\right)^{-1}\,
.
\end{equation}

There is a $(d,k)$-dimensional analogue of the fixed point differential
Equation~\ref{equ_gfn}:
\begin{prop}
\label{prop_equ_diff_natdk}
The exponential generating function $\GFNdk$ of generalized non-ambiguous trees
satisfies the following differential equation
\begin{equation}
\label{equ_gfndk}
\GFNdk
:= 
\sum_{\N \in \NATdk^*}
	\prod_{i=1}^d
	\frac{
		x_i^{\wi_i(\N)-1}
	}{
		( \wi_i(\N) -1) !
	}
=
\prod_{\dir \in \edir} \left( 1 + \int_{\dir} \GFNdk  \right)
\end{equation}
\end{prop}

\begin{proof}
The method is analogue to the method of Section \ref{differential}, and goes through
the use of a $\binom{d}{k}$-linear map and a pumping function for $\binom{d}{k}$-ary trees.
\end{proof}

%

The family of differential equations defined by Equation~\ref{equ_gfndk} can be 
rewritten using differential operators instead of primitives.
We need to introduce the function
$
\LGFNdk
=
\int_{\{1, \dots, d\}}
	\GFNdk
+
\sum_{\dir \in \edir[d][d-k]}
	x_{\dir}
$.
Then, we show that $\LGFNdk$ satisfies the following differential equations:
\begin{prop}
The differential equation satisfied by $\LGFNdk$ is 
$
\partial_1 \ldots \partial_d \LGFNdk
=
\prod_{\dir \in \edir[d][d-k]}
	\partial_{\dir} \LGFNdk
.
$
\end{prop}

In the generic case, we are not able to solve those differential equations.
We know that setting a variable $x_d$ to $0$ gives the generating function 
of NATs of lower dimension.
\begin{prop}
Let $d > k \ge 1$, then
$
\left.
	\GFN_{d,k}
\right|_{x_d=0}
=
\GFN_{d-1,k}
.
$
\end{prop}

For some specific values of $d$ and $k$ we have (at least partial) results.
\begin{prop}
\label{prop_sol_d_d-1}
Let $k=d-1$, if we know a particular solution $s( x_1, \ldots, x_d)$ for
$$
\partial_1 \ldots \partial_d \LGFNdk[d][d-1]
=
\partial_{1} \LGFNdk[d][d-1] \times \ldots \times \partial_{d} \LGFNdk[d][d-1] 
$$
then, for any function $s_1(x_1), \ldots, s_d(x_d)$, the function
$s( s_1(x_1), \ldots, s_d(x_d))$ is also a solution.
\end{prop}

\begin{prop}
\label{prop_sol_d_1}
Some non trivial rational functions are solutions of
$
\partial_1 \ldots \partial_d \LGFNdk[d][1]
=
\prod_{\dir \in \edir[d][d-1]}
	\partial_{\dir} \LGFNdk[d][1]
.
$
\end{prop}

\begin{proof}[(sketch)]
We  define
$
\LGFN_{(i)} = \partial_{\dir} \LGFNdk[d][1]
$
where 
$i \in \i{1}{d}$
and
$\dir = \i{1}{d} \setminus \{i\}$.
We get the relation
$
\partial_i \LGFN_{(i)}
=
\prod_{j=1}^d \LGFN_{(j)}
$
and then
$
\prod_{i=1}^d
\partial_i \LGFN_{(i)}
=
\prod_{i=1}^d \LGFN_{(i)}^d
.
$
To obtain a particular solution, we just need to identify, in the previous 
equation, the term $\partial_i \LGFN_{(i)}$ to the term $\LGFN_{(i)}^d.$
We thus obtain some non trivial solutions for our equation,
which are rational functions.
\end{proof}

Since dimension $(2,1)$ is the unique case where Proposition~\ref{prop_sol_d_d-1}
and Proposition~\ref{prop_sol_d_1} can be applied at the same time,
and the computation of $\GFNdk[\d][\d]$ is straightforward,
we have the following proposition.
\begin{prop}
We have the closed formulas:
$\GFNdk[2][1] = \GFN$ and
$\GFNdk[d][d]=\sum_{n \ge 0} \frac{(x_1\cdot \ldots \cdot x_d)^n}{(n!)^d}.$
\end{prop}

\noindent We see $\GFNdk[d][d]$ as is a kind of generalized Bessel
function because
$\GFNdk[2][2](x/2, -x/2) = J_0(x)$ where $J_\alpha$ is the classical Bessel 
function.
This supports our feeling that the general case leads to serious difficulties.

\subsection{Geometric interpretation} 
\label{geom_gnat}

As for non-ambiguous trees, we can give a geometric definition
of non-ambiguous trees of dimensions $(d,k)$ as follows.
We denote by $(e_1,\dots,e_d)$ the canonical basis of $\mathbb{R}^\d$ and
$(\X_1,\ldots,\X_\d)$ its dual basis, i.e. $\X_i$ is $\mathbb{R}$-linear
$\X_i(e_i)=\delta_{i, j}$.
Let $P\in\mathbb{R}^\d$ and $\dir=\{i_1,\ldots,i_k\}$ a $(d,k)$-direction,
we call \emph{cone of origin $P$ and direction $\dir$} the set of points
$C(P, \pi) := \{P + a_1e_{i_1} + \cdots + a_ke_{i_k}\ |\ (a_1,\ldots,a_k)\in\mathbb{N}^k\}.$
\begin{defi}\label{def_gnat_geo}
A \emph{geometric non-ambiguous tree} of dimension $(\d,\k)$ and box $\gs$
is a non empty set $\V$ of points of $\mathbb{N}^\d$ such that:
\begin{enumerate}
\item\label{gnat_geo_box} $\i{1}{\wi_1}\times\cdots\times\i{1}{\wi_\d}$ is the smallest box containing $\V$,
\item\label{gnat_geo_root} $\V$ contains the point $(\wi_1,\ldots,\wi_\d)$, which is called the \emph{root},
\item\label{gnat_geo_cone} For $P\in \V$ different from the root, there exists a unique
$(d,k)$-direction $\dir=\{i_1,\ldots,i_k\}$ such that the cone $c(P, \dir)$
 contains at least one point different
from $P$. We say that $P$ is of type $\dir$.
\item\label{gnat_geo_affine} For $P$ and $P'$ two points of $\V$ belonging
to a same affine space of direction $\Vect(e_{i_1},\ldots,e_{i_k})$,
then, either $\forall j\in\i{1}{\k}\text{, }X_{i_j}(P)>X_{i_j}(P')$, or
$\forall j\in\i{1}{\k}\text{, }X_{i_j}(P')>X_{i_j}(P).$
\item\label{gnat_geo_distinct_and_interval} For each $i\in\i{1}{\d}$,
for all $l\in\i{1}{\wi_i-1}$, the affine hyperplane $\{x_i=l\}$
contains exactly one point of type $\dir$ such that $i\in\dir$.
\end{enumerate}
\end{defi}

\begin{prop}
There is a simple bijection between the set of geometric non-ambiguous tree of
box $\gs$ and the set of non-ambiguous tree of
geometric size $\gs$.
\end{prop}

\acknowledgements
The authors thank Samanta Socci for fruitful discussions which were
the starting point of the generalisation of non-ambiguous trees.

This research was driven by computer exploration using the open-source software \texttt{Sage}~\cite{sage} and its algebraic combinatorics features developed by the \texttt{Sage-Combinat} community~\cite{Sage-Combinat}.

\nocite{*}
\bibliographystyle{alpha}
\bibliography{biblio}
\end{document}